\documentclass[a4paper,12pt]{article}
\usepackage{amsthm}
\usepackage{amsfonts}
\usepackage{amsmath}
\usepackage{amssymb}
\usepackage[pdftex]{graphicx}
\usepackage{hyperref}
\usepackage{datetime}
\usepackage{diagrams}
\usepackage{commath}
\usepackage{float}
\usepackage{graphicx}
\usepackage{bm}
\restylefloat{table}

\newtheorem{theorem}{Theorem}[section]

\newtheorem{proposition}[theorem]{Proposition}
\newtheorem{lemma}[theorem]{Lemma}
\newtheorem{corollary}[theorem]{Corollary}

\newcommand{\C}{\mathbb{C}}

\newcommand{\R}{\mathbb{R}}

\newcommand{\diag}{\text{diag}}
\newcommand{\Gtwo}{\mathbf{G}_2}

\newcommand{\Imo}{\operatorname{Im}\mathbb{O}}
\makeatletter
\newcommand*\bigcdot{\mathpalette\bigcdot@{.5}}
\newcommand*\bigcdot@[2]{\mathbin{\vcenter{\hbox{\scalebox{#2}{$\m@th#1\bullet$}}}}}
\makeatother

\begin{document}

\title{Almost positive curvature on an irreducible compact rank $2$ symmetric space}

\author{Jason DeVito and Ezra Nance}

\date{}

\maketitle

\begin{abstract}

A Riemannian manifold is said to be almost positively curved if the sets of points for which all $2$-planes have positive sectional curvature is open and dense.  We show that the Grassmannian of oriented $2$-planes in $\mathbb{R}^7$ admits a metric of almost positive curvature, giving the first example of an almost positively curved metric on an irreducible compact symmetric space of rank greater than $1$.  The construction and verification rely on the Lie group $\Gtwo$ and the octonions, so do not obviously generalize to any other Grassmannians.

\end{abstract}

\section{Introduction} 
\label{intro}

\noindent The collection of closed simply connected manifolds admitting a Riemannian metric of positive sectional curvature forms an intriguing class.  Apart from spheres and projective spaces, all such known examples occur only in dimensions $6,7,12,13$ and $24$ \cite{Es1,AW,Wa,Baz1,De,GVZ,Ber}.  However, there are very few known obstructions.  For example, if $M$ is a closed simply connected manifold admitting a non-negatively curved metric, then there is no known obstruction to $M$ admitting a positively curved metric.

If one relaxes the positivity condition, examples become easier to construct.  For example, one may ask for a non-negatively curved metric on $M$ for which every $2$-plane at a single point is positively curved.  Such an $M$ is said to be \textit{quasi-positively curved}.  One may also ask for more: that the set of points for which all $2$-planes are positively curved be open and dense.  This property is referred to as \textit{almost positive curvature}.  Examples of manifolds admitting metrics of quasi-positive or almost positive curvature are more abundant \cite{GrMe1,PW2,Wi,DDRW,Ta1,KT,EK,Ke1,Ke2} and include several families in arbitrarily high dimensions.

In \cite{Zi2}, one finds a generalization of the classical Hopf conjecture: that no compact symmetric space of rank $2$ or more admits a metric of positive curvature.  In \cite{Wi}, Wilking shows that the reducible rank $2$ symmetric spaces $S^3\times S^2$ and $S^7\times S^6$ admit an almost positively curved metric, showing that the hypothesis of the conjecture cannot be weakened to almost positive curvature.  We provide the first irreducible counterexample to the weakened conjecture.

\begin{theorem}\label{main}  The Grassmannian of oriented $2$-planes in $\mathbb{R}^7$, $Gr_2\!\left(\mathbb{R}^7\right)$, admits an almost positively curved metric invariant under an $SU(3)$ action of cohomogeneity two.  Further, this metric descends to an almost positively curved metric on the Grassmannian of unoriented $2$-planes in $\mathbb{R}^7$.
\end{theorem}

In \cite{KT}, Kerr and Tapp show the homogeneous space $\Gtwo/U(2)$, with ${U(2)\subseteq SU(3)\subseteq \Gtwo}$, admits a metric of quasi-positive curvature.  We recall that $\Gtwo/U(2)$ is known to be diffeomorphic to the Grassmannian of oriented $2$-planes in $\mathbb{R}^7$, $Gr_2\!\left(\mathbb{R}^7\right)$, see, for example, \cite[Lemma 1.1]{Ker1}.  This example is, in fact, the first metric of quasi-positive curvature on an irreducible symmetric space of rank bigger than $1$.  We show their quasi-positively curved metric is actually almost positively curved.

We show the $SU(3)$ action is by cohomogeneity $2$ by finding an explicit $2$-dimensional disc in $\Gtwo/U(2)$ which meets every orbit, see Proposition \ref{niceform}.  Unfortunately, because we rely on the octonions and $\Gtwo$ for the construction and verification of the metric properties, the method of proof does not seem to extend to any other irreducible symmetric spaces of rank $2$ or more.

We are actually able to obtain an explicit description of the set of points having at least one zero-curvature plane.

\begin{theorem}\label{main2}
An element $g =(g)_{ij}\in \Gtwo\subseteq SO(7)$ projects in $G_2/U(2)$ to a point having at least one zero-curvature plane iff $g_{12} = g_{13} = 0$ or $g_{11} = 0$.

\end{theorem}

We let $Z_1 = \{g\in \Gtwo: g_{12} = g_{13} = 0\}$ and $Z_2 = \{g\in \Gtwo: g_{11} = 0\}$.  Then we have the following description of the topology of the image of $Z_1$ and $Z_2$ in $\Gtwo/U(2)$.

\begin{theorem}\label{main3} The projection of $Z_1$ to $\Gtwo/U(2)$ has image diffeomorphic to the Grassmannian of oriented $2$-planes in $\mathbb{R}^6$, $Gr_2\!\left(\mathbb{R}^6\right),$ while the projection of $Z_2$ has image diffeomorphic to $\mathbb{C}P^2\times S^5$.  The intersection of the projections is diffeomorphic to the Aloff-Wallach space $W_{1,-1}=W_{1,0}$.

\end{theorem}

We note that $W_{1,-1}$ is the unique Aloff-Wallach space which does not admit a homogeneous metric of positive curvature.

The outline of this paper is as follows.  Section \ref{G2oct} will review necessary facts about the octonions and $\Gtwo$, proving Proposition \ref{niceform}. 

In Section \ref{metric}, we use Cheeger deformations and Wilking's doubling trick to construct our metric.  More precisely, if $\langle \cdot,\cdot \rangle_1$ denotes the result of Cheeger deforming a bi-invariant metric on $\Gtwo$ in the direction of $SU(3)$, we equip $\Gtwo\times \Gtwo$ with $\langle \cdot , \cdot \rangle_1 + \langle \cdot, \cdot \rangle_1$ and induce a metric on $\Gtwo/U(2)$ as the submersion metric $\Gtwo\times \Gtwo\rightarrow \Delta \Gtwo\backslash \Gtwo\times \Gtwo/1\times U(2)\cong \Gtwo/U(2)$.  We note that our construction is somewhat different from those found in, e.g., \cite{Wi,Ke2} in that $(\Gtwo,SU(3))$ is not a symmetric pair.  Nevertheless, we show that the ``usual'' curvature conditions match those of a symmetric pair, see Proposition \ref{equiv}.

In Section \ref{apc}, we complete the proof of Theorem \ref{main} and \ref{main2} by reducing the problem to a direct calculation using the $2$-dimensional disc of Proposition \ref{niceform}.  Finally, in Section \ref{top}, we compute the diffeomorphism type of the points having at least one zero-curvature plane, proving Theorem \ref{main3}.
\section{\texorpdfstring{$\Gtwo$}{G2} and the Octonions}

\label{G2oct}

Much of the background can be found in \cite{Mur}; we use the conventions found in \cite{Ke2}.  The octonions $\mathbb{O}$ are a non-associative normed division algebra of dimension $8$ over $\R$.  The octonions are alternative, meaning that the subalgebra generated by any two elements of $\mathbb{O}$ is associative.  A general octonion may be expressed in the form $a + b\ell$ where $a,b\in \mathbb{H}$, the set of quaternions.  Multiplication is defined by the Cayley-Dickson construction and is given by $$(a+b\ell)(c+d\ell) = (ac - \overline{d}\,b) + (da + b\,\overline{c})\ell.$$

We use the ordered basis $\{i,j,k,\ell, i\ell, j\ell, k\ell\}$ of $\Imo$, which we declare to be orthonormal.  All of our $7\times 7$ matrices will be expressed with respect to this basis.

We have the following multiplication table in the form (row)(column).

\begin{center}


\begin{table}[ht]

\begin{center}

\begin{tabular}{cccccccc}
 
\multicolumn{1}{c}{} & $\bm{i}$ & $\bm{j}$ & $\bm{k}$ & $\bm{\ell}$ & $\bm{i\ell}$ & $\bm{j\ell}$ & $\bm{k\ell}$ \rule{0pt}{4ex} \\

\rule{0pt}{6ex} 

$\bm{i}$ & $-1$ & $k$ & $-j$ & $i\ell$ & $-\ell$ & $-k\ell$ & $j\ell$  \\ 
\rule{0pt}{4ex}

$\bm{j}$ & $-k$ & $-1$ & $i$ & $j\ell$ & $k\ell$ & $-\ell$ & $-i\ell$ \\ 
\rule{0pt}{4ex}

$\bm{k}$ & $j$ & $-i$ & $-1$ & $k\ell$ & $-j\ell$ & $i\ell$ & $-\ell$\\ 
\rule{0pt}{4ex}

$\bm{\ell}$ & $-i\ell$ & $-j\ell$ & $-k\ell$ & $-1$ & $i$ & $j$ & $k$\\ 
\rule{0pt}{4ex}

$\bm{i\ell}$ & $\ell$ & $-k\ell$ & $j\ell$ & $-i$ & $-1$ & $-k$ & $j$\\ 
\rule{0pt}{4ex}

$\bm{j\ell}$ & $k\ell$ & $\ell$ & $-i\ell$ & $-j$ & $k$ & $-1$ & $-i$\\ 
\rule{0pt}{4ex}

$\bm{k\ell}$ & $-j\ell$ & $i\ell$ & $\ell$ & $-k$ & $-j$ & $i$ & $-1$\\ 
\rule{0pt}{4ex}

\end{tabular}

\caption{Multiplication table for Cayley numbers}\label{table:cayleymult}

\end{center}

\end{table}

\end{center}

The Lie group $G=\Gtwo$ is, by definition, the set of all automorphisms of the octonions.  That is, $$\Gtwo = \{A\in Gl_8(\mathbb{R}): A(xy) = A(x)A(y)\text{ for every }x,y\in \mathbb{O}\}.$$  One can show that $G\subset SO(8)$ and, using the fact that every element in $G$ fixes $1$ and therefore the imaginary octonions $\Imo$, that $G$ is naturally a subgroup of $SO(7)$.

Given $g\in G\subseteq SO(7)$, the following notations will be used:

\renewcommand{\labelitemi}{\textendash}
\begin{itemize}
\item  $g_{mn}$ refers the entry in row $m$ and column $n$
\item  $g_{\bullet n}$ refers to the $n$-th column
\item  $g_{m\bullet}$ refers to the $m$-th row
\item  $g_{\bullet s} g_{\bullet t}$ refers to octonionic multiplication of of $g_{\bullet s}$ and $g_{\bullet t}$ each interpreted as elements of $\Imo$
\item  $(g_{\bullet s}) \bigcdot (g_{\bullet t})$ refers to the usual Euclidean dot product.
\end{itemize}

The elements of $G$ have the following characterization, a proof of which can be found in \cite[pg. 186]{Mur}.

\begin{theorem}\label{G2desc}  Suppose $e_1, e_2, e_3\in \Imo$ are orthonormal and that, in addition, $e_3$ is perpendicular to $e_1 e_2$. Then there is a unique $g\in \Gtwo$ with $g(e_1) = i$, $g(e_2) = j$, and $g(e_3) = \ell$.
\end{theorem}

The characterization of $G$ as the automorphisms of the octonions allows us to recognize when a matrix in $SO(7)$ is actually in $G$:  the matrix  must have columns $g_{\bullet 1},...,g_{\bullet 7}$ of the form $$\begin{bmatrix}g_{\bullet 1} & g_{\bullet 2} & g_{\bullet 1}\, g_{\bullet 2} & g_{\bullet 4} & g_{\bullet 1}\, g_{\bullet 4}& g_{\bullet 2} \,g_{\bullet 4} & (g_{\bullet 1}\, g_{\bullet 2})g_{\bullet 4}\end{bmatrix}.$$ Further, since the transpose, which is also the inverse, of a matrix in $G$ is in $G$, the same form holds for the rows $g_{1\bullet}, ..., g_{7\bullet}$.  In short, an element of $g\in G$ is determined by the columns $g_{\bullet 1}, g_{\bullet 2},$ and $g_{\bullet 4}$, and also by the rows $g_{1\bullet}, g_{2\bullet}$, and $g_{4\bullet}$.

The group $G$ has several important subgroups.  For example, if $$K = \{g\in G: g(i) = i\},$$ then $K$ is isomorphic to $SU(3)$ \cite[Theorem 5.5]{Ad2}.  Together with Theorem \ref{G2desc}, it is now easy to see that $G/K$ is diffeomorphic to $S^6$.

Now, set $H = \{g\in G: g\text{ preserves the oriented }jk\text{-plane}\}$.  Note that $H\subseteq SU(3)$:  since the action by $g$ on the $jk$-plane is simply rotation, $$g(j)g(k) = (\cos\theta j + \sin\theta k)(-\sin\theta j + \cos\theta k) = i.$$  Then according to \cite[Lemma 1.1]{Ker1}, $H$ is isomorphic to $U(2)\subseteq SU(3)$ and $G/H$ is diffeomorphic to $Gr_2\!\left(\mathbb{R}^7\right)$.

Next, consider the element $\sigma = \diag(-1,1,-1,1,-1,1,-1)\in N(K)\subseteq G$, where $N(K)$ denotes the normalizer of $K$ in $G$.    We can enlarge $H$ to $H' = H\cup \sigma H$.  Then, in a similar fashion, it can be shown that $G/H'$ is diffeomorphic to the Grassmannian of unoriented $2$-planes in $\mathbb{R}^7$.

We also have an alternative description of $H$.  Consider the action of $S^1\times Sp(1)$ on $\Imo = \operatorname{Im}\mathbb{H}\oplus \mathbb{H}\ell$ given by $(z,q)*(a+b\ell) = za\overline{z} + (zb\overline{q})\ell$.  The one can show the kernel of this action is generated by $(-1,-1)$ and that the action is by automorphisms of $\mathbb{O}$.  Since the $S^1\times Sp(1)$ action preserves the oriented $jk$-plane, we have an embedding $U(2)\rightarrow H$, which, must therefore be an isomorphism.

We will eventually see that the action of $H'\times K$ on $G$ by $(h,k)\ast g = hgk^{-1}$ is isometric.  With this in mind, the following proposition will be the key to understanding points having zero-curvature planes.

\begin{proposition}\label{niceform}
Consider the action of $H'\times K$ on $G$ given by $(h,k)\ast(g) = hgk^{-1}$.  Then every orbit passes through a unique point of the form \begin{equation}\begin{bmatrix}\label{matrix} \cos\theta & \sin\theta & 0 &0 & 0 & 0 & 0 \\ -\cos\phi \sin\theta & \cos\phi \cos\theta & 0 & -\sin\phi & 0 & 0 & 0\\ 0 & 0 & \cos\phi & 0 & -\sin\phi \cos\theta & -\sin\phi \sin\theta & 0 \\ -\sin \phi \sin \theta &  \sin\phi \cos\theta & 0 & \cos\phi & 0 & 0 & 0\\ 0 & 0 & \sin\phi & 0 & \cos\phi \cos\theta & \cos\phi \sin\theta & 0 \\ 0 & 0 & 0 & 0 & -\sin\theta & \cos\theta & 0 \\ 0 & 0 & 0 &0 & 0 & 0 & 1 \end{bmatrix}\end{equation} where $0\leq \theta ,\phi\leq \pi/2$.

\end{proposition}

We use the notation $\mathcal{F}\subseteq \Gtwo$ to denote the subset of points having this form.

\begin{proof}  Let $g = (g)_{ij}\in G\subseteq SO(7)$.  If $g_{11} < 0$, we initially apply the element $\sigma \in H'\times\{1\}$.  Now, from the above description of $H$ as a quotient of $S^1\times Sp(1)$, we see that each element of $H$ is a block diagonal matrix of the form $\diag(1, R(\alpha), A)$ where  $$R(\alpha) = \begin{bmatrix} \cos\alpha & \sin\alpha \\ -\sin\alpha & \cos\alpha\end{bmatrix},$$ and $A$ is a $4\times 4$ matrix in a $U(2)\subseteq SO(4)$.  Similarly, as every element of $K$ fixes $i$, each element has the form $\diag(1, B)$ where $B$ is an element of an $SU(3)\subseteq SO(6)$.  In particular, the $H\times K$ action on $G$ fixes the $g_{11}$ coordinate.  We uniquely define $\theta\in [0,\pi/2]$ via $g_{11} = \cos(\theta)$.  Note that if $\cos(\theta) =  1$, then $g\in SU(3)$, so clearly lies in the orbit of the identity, having the form of \eqref{matrix}.  Thus, we may assume $\theta\in (0,\pi/2]$ for the remainder of this proof.

We now consider the subaction by $H\times\{1\}$ on the first column $g_{\bullet 1}$ of $g$.  Note that $R(\alpha)$ acts by rotations on $\begin{bmatrix} g_{21} & g_{31}\end{bmatrix}^t,$ so we see that the length $g_{21}^2 + g_{31}^2$ is an invariant under the $H$ action.  In particular, since ${0\leq g_{21}^2 + g_{31}^2\leq \sin^2\theta}$, we may uniquely define $\phi$ by the equation $\cos\phi\sin\theta = \sqrt{g_{21}^2 + g_{31}^2}$.
Now, by picking $\alpha$ appropriately, we may rotate the vector $\begin{bmatrix} g_{21}& g_{31}\end{bmatrix}^t$ to the vector $\begin{bmatrix} -\cos\phi \sin\theta &  0  \end{bmatrix}^t$.  After this, we may then choose a new element of $H$ with $\alpha = 0$, that is, an element of $SU(2)\subseteq H$, to modify the rest of $g_{\bullet 1}$ to have the form of $\eqref{matrix}$.  This follows because the only faithful representation of $SU(2)$ on $\mathbb{R}^4$ is transitive on each sphere of fixed radius centered at the origin.

We next consider the subaction by $\{1\}\times K$ on the first row $g_{1\bullet}$ of $g$.  Since each element of $K$ has the form $\diag(1,B)$, $g_{\bullet 1}$ is fixed.  In addition, $K$ acts on the last $6$ coordinates of $g_{1\bullet}$, $(g_{12}, g_{13},..., g_{17})$ by some real representation.  There is a unique non-trivial $6$-dimensional real representation of $SU(3)$ coming from the identification of $\C^3$ with $\R^6$, and this representation acts transitively on the sphere of any fixed radius centered at the origin.  In particular, the $\{1\}\times K$ orbit through $g$ contains a point whose first row is as in $\eqref{matrix}$.

\

We next consider the subgroup of $K$ given by those elements which fix $i$, $j$, and $k$.  This subgroup is isomorphic to $SU(2)$; in fact, it is the $SU(2)$ in $H$.  A matrix in this subgroup has the form $\diag(1,1,1,A)$, so right multiplication by it will not modify $g_{\bullet 1}$, $g_{\bullet 2}$ or $g_{\bullet 3}$.  However, as done previously, we may use such an element to move the vector $(g_{24}, g_{25}, g_{26}, g_{27})$ to one of the form $(\lambda, 0,0,0)$ for some non-positive real number $\lambda$.

Now, $g_{22}$ is determined by the fact $(g_{1\bullet})\bigcdot (g_{2\bullet}) = 0$; likewise, $g_{32}, g_{42}, g_{52}, g_{62},$ and $g_{72}$.  Since we now know the first two columns, octonionic multiplication gives the third.  At this point, $\lambda = -\sin \phi$ is now determined since the length of the $g_{2\bullet}$ is $1$.  This completes the determination of the second row and thus, of the third row as well.

If $\sin \phi = 0$, then we see we can pick a new element of $H$ which moves the column $\begin{bmatrix} g_{44} & g_{54} & g_{64} & g_{74}\end{bmatrix}^t$ to $\begin{bmatrix} 1 & 0 & 0 &0\end{bmatrix}^t$.  This finishes the determination of column $4$, and hence all the rest of the entities.

\

On the other hand, if $\sin \phi \neq 0$, then the equation $0 = (g_{\bullet 1})\bigcdot(g_{\bullet 4})$ gives  $g_{44} = \cos \phi$.  Since the row $g_{4\bullet}$ has unit length, this now forces all the remaining unknown entries in $g_{4\bullet}$ to be $0$, finishing the determination of $g_{4\bullet }$.  The rest of the entries are now determined since $g\in \Gtwo$.

\end{proof}

As a corollary to the proof, we see that for $g\in G$, $|g_{11}|$, the length of $(g_{21}, g_{31})$, and the length of $(g_{41}, g_{51}, g_{61}, g_{71})$ determine the $H'\times K$ orbit.

\

\
We now describe the Lie algebras of $H\subseteq K\subseteq G\subseteq SO(7)$.  Since we are following the conventions of \cite{Ke1}, $\mathfrak{g} = \mathfrak{g}_2 $ consists of all real matrices of the form  \begin{equation}\label{g2form}\begin{bmatrix} 0 & x_1 + x_2 & y_1 + y_2 & x_3 + x_4 & y_3 + y_4& x_5 + x_6 & y_5 + y_6\\ -(x_1 + x_2) & 0 & z_1 & - y_5 & x_5 & -y_3 & x_3\\ -(y_1 + y_2) & -z_1 &0 & x_6 & y_6 & -x_4 & -y_4\\ -(x_3 + x_4) & y_5 & -x_6 & 0 & z_2 & y_1 & -x_1\\ -(y_3 + y_4) & -x_5 & -y_6 & -z_2 & 0 & x_2 & y_2\\ -(x_5 + x_6) & y_3 & x_4 & -y_1 &  -x_2 & 0 & z_1 + z_2\\ -(y_5  + y_6)& -x_3 & y_4 & x_1 & -y_2 & -(z_1 + z_2) & 0 \end{bmatrix} .\end{equation}

Then the subalgebra $\mathfrak{k} = \mathfrak{su}(3)$ consists of those matrices in $\mathfrak{g}$ whose first row and first column vanish, and the subalgebra $\mathfrak{h} = \mathfrak{u}(2)$ has the additional constraint that $x_3 = x_4 = x_5 = x_6 = 0$ and similarly for $y$.  With respect to the bi-invariant metric $\langle X, Y\rangle_0 = -\operatorname{Tr}(XY)$, we have an orthogonal splitting $\mathfrak{g} = \mathfrak{k}\oplus \mathfrak{p}$.  Writing $X\in \mathfrak{g}$ in the form \eqref{g2form}, a simple calculation shows the projection of $X\in\mathfrak{g}$ to $\mathfrak{k}$ is \begin{equation}\label{su3form}  \frac{1}{2} \begin{bmatrix} 0 & 0 & 0 & 0 & 0 & 0 & 0 \\ 0 & 0 & 2z_1 & y_6 - y_5 & x_5 - x_6 & y_4 - y_3 & x_3 - x_4\\ 0 & -2z_1 & 0 & x_6-x_5 & y_6-y_5 & x_3 - x_4 & y_3 - y_4 \\ 0 & y_5-y_6 & x_5-x_6 & 0 & 2z_2 & y_1-y_2 & x_2 - x_1\\ 0 & x_6-x_5 & y_5 - y_6 & -2z_2 & 0 & x_2 - x_1 & y_2 - y_1\\ 0 & y_3 - y_4 & x_4 - x_3 & y_2 - y_1 & x_1 - x_2 & 0 & 2z_1 + 2z_2\\ 0 & x_4 - x_3 & y_4 - y_3 & x_1 - x_2 & y_1 - y_2 & -2z_1 - 2z_2 & 0\end{bmatrix}\end{equation} and that the projection to $\mathfrak{p}$ is \begin{equation} \label{pform} \frac{1}{2}\begin{bmatrix} 0 & 2(x_1 + x_2) & 2(y_1 + y_2) & 2(x_3 + x_4) & 2(y_3 + y_4) & 2(x_5 + x_6) & 2(y_5 + y_6)\\ -2(x_1 + x_2) & 0 & 0 & -(y_5 + y_6) & x_5 + x_6 &  -(y_3 + y_4) & x_3 + x_4\\ -2(y_1 + y_2) & 0 & 0 & x_5 + x_6 & y_5 + y_6 & -(x_3 + x_4) & -(y_3 + y_4)\\ -2(x_3 + x_4) & y_5 + y_6 & -(x_5 + x_6) & 0 & 0 & y_1 + y_2 & -(x_1 + x_2)\\ -2(y_3 + y_4) & -(x_5 + x_6) & -(y_5 + y_6) & 0 & 0 & x_1 + x_2 & y_1 + y_2\\ -2(x_5 + x_6) & y_3 + y_4 & x_3 + x_4 & -(y_1 + y_2) & -(x_1 + x_2) & 0 & 0 \\ -2(y_5 + y_6) & -(x_3 + x_4) & y_3 + y_4 & x_1 + x_2 & -(y_1 + y_2) & 0 & 0 \end{bmatrix} .\end{equation}

From this, we see that the entries of a matrix in $\mathfrak{p}$ are determined by the top row.  Because of this, we will sometimes abuse notation and refer to such a matrix by the ordered $6$-tuple $(x_1 + x_2, y_1 + y_2,.... ) \in \mathbb{R}^6$.

\section{Construction of the Metric}\label{metric}

As is shown in \cite{Ta1}, the metric we will use is, up to scaling, isometric to the metric considered by Kerr and Tapp \cite{KT}.  We construct our metric via Cheeger deformations \cite{Ch1} and Wilking's doubling trick \cite{Wi}.

Let $G$ denote an arbitrary compact Lie group with a closed subgroup $K\subseteq G$.  We use the notation $\mathfrak{k}\subseteq \mathfrak{g}$ to denote the Lie algebras of $K$ and $G$.

Let $\langle X,Y\rangle_0$ denote a bi-invariant metric on $G$; for $G =\Gtwo$, we use $\langle X,Y\rangle=-\operatorname{Tr}(XY)$.  We let $\mathfrak{p}\subseteq \mathfrak{g}$ denote the orthogonal complement to $\mathfrak{k}$ with respect to $\langle \cdot, \cdot \rangle_0$ and we use the notation $X = X_\mathfrak{k} + X_\mathfrak{p}$ to refer to the projections of $X$ onto $\mathfrak{k}$ and $\mathfrak{p}$.

 We let $\langle \cdot, \cdot \rangle_1$ denote the left $G$-invariant, right $K$-invariant metric obtained via Cheeger deforming $\langle \cdot, \cdot \rangle_0$ in the direction of $K$.   That is, we first equip $G\times K$ with the metric $\langle \cdot, \cdot\rangle_0 + t \langle \cdot, \cdot\rangle_0|_K$ for a fixed parameter $t>0$.  The group $K$ acts isometrically on $G\times K$ via $k\ast(g_1,k_1) = (g_1k^{-1}, kk_1)$.  One can easily verify the map $G\times K\rightarrow G$ given by $(g_1,k_1)\mapsto g_1k_1$ descends to a diffeomorphism $G\times_K K\cong G$, which we use to transport the submersion metric on $G\times_K K$ to $G$, obtaining the metric $\langle \cdot ,\cdot \rangle_1$.  One can also verify (see, for example, \cite{Ke1}) that $\langle X,Y\rangle_1 = \langle X, \phi(Y)\rangle_0$ where $\phi(Y)= \frac{t}{t+1}Y_\mathfrak{k} + Y_\mathfrak{p}$.  From O'Neill's formula \cite{On1} for curvature of a Riemannian submersion, together with the fact that bi-invariant metrics are always non-negatively curved, we see that $\langle \cdot, \cdot \rangle_1$ is non-negatively curved.

We also point out that $G\times K$ acts by isometries on $G\times_K K$ by $(g,k)\ast[g_1, k_1] = [g g_1, k_1 k^{-1}]$.  In particular, the metric $\langle \cdot, \cdot\rangle_1$ is left $G$-invariant and right $K$-invariant, as claimed.  In fact, the isometry group is often larger.

\begin{proposition}\label{isomgroup}

If $N(K)$ denotes the normalizer of $K$ in $G$, then $\langle \cdot, \cdot \rangle_1$ is right $N(K)$-invariant.

\end{proposition}

\

\begin{proof}

Let $n\in N(K)$.  Because left multiplication by $n$ is an isometry, right multiplication by $n$ is an isometry iff conjugation by $n$ is also an isometry.

For any $n\in N(K)$, $Ad_n:\mathfrak{g}\rightarrow \mathfrak{g}$ preserves $\mathfrak{k}$ because, for any curve $\gamma$ in $K$, $n\gamma(t)n^{-1}\in K$.  Because right multiplication by $n$ is an isometry with respect to the bi-invariant metric $\langle \cdot, \cdot \rangle_0$, we see that $Ad_n$ also preserves $\mathfrak{p}$.  Hence, $Ad_n(\phi(Y) = \phi(Ad_n(Y))$.

But then \begin{align*} \langle Ad_n(X), Ad_n(Y)\rangle_1 &= \langle Ad_n(X), \phi(Ad_n(Y))\rangle_0 \\ &=\langle Ad_n(X), Ad_n(\phi(Y))\rangle_0\\ &= \langle X,\phi(Y)\rangle_0\\ &= \langle X, Y\rangle_1.\end{align*}

\end{proof}

From O'Neill's formulas \cite{On1}, a zero-curvature plane in $(G, \langle \cdot, \cdot \rangle_1)$ must lift to a horizontal zero-curvature plane in $G\times K$.   In addition, according to Tapp \cite{Ta2}, a horizontal zero-curvature plane in $G\times K$ will always project to a zero-curvature plane in $G$. One can show (see, for example, \cite{Ke1}) that the lift of a vector $X \in \mathfrak{g}$ is given by $\left( \phi(X), -\frac{1}{t+1} X_{\mathfrak{k}}\right).$  Since the metric on $G\times K$ is a product of bi-invariant metrics, we get the following proposition.

\begin{proposition}\label{0curv} A tangent plane given by $\operatorname{span}\{X,Y\}\subseteq T_e G$ has zero curvature with respect to $\langle \cdot, \cdot\rangle_1$ iff $[\phi(X), \phi(Y)] = 0$ and $[X_\mathfrak{k}, Y_\mathfrak{k}] = 0$.

\end{proposition}

When Cheeger deformations have been used previously, the pair $(G,K)$ has always been symmetric, meaning $[\mathfrak{p},\mathfrak{p}]\subseteq \mathfrak{k}$.  This allows one to conclude, under the assumption $[\phi(X),\phi(Y)] = 0$, that ${[X_\mathfrak{k}, Y_{\mathfrak{k}}] = 0}$ iff $[X_\mathfrak{p}, Y_\mathfrak{p}] = 0$.  The pair $(\Gtwo, SU(3))$ is not symmetric, but nevertheless, the same conclusion holds.  To see this, we first need a lemma.

\begin{lemma}\label{almostsym}

Suppose $(G,K) = (\Gtwo, SU(3))$.  For $X,Y\in \mathfrak{p}\subseteq  \mathfrak{g}_2$, $[X, Y] = 0$ iff $[X, Y]_{\mathfrak{k}} = 0$.

\end{lemma}

\begin{proof}  The forward implication is trivial, so we focus on the reverse direction.  We may also obviously assume $X\neq 0$ and that $Y$ is perpendicular to $X$.

Now, for $g\in K$, we note that $Ad_g([X,Y]_\mathfrak{k}) = [Ad_g X, Ad_g Y]_\mathfrak{k}$.  Since the adjoint action of $K$ on $\mathfrak{p}$ is equivalent to the standard action of $SU(3)$ on $\mathbb{R}^6$, and this action is transitive on spheres centered at the origin, we may assume without loss of generality that $X = (x_1,0,...,0)\in \mathfrak{p}$ with $x_1\neq 0$.  Here, we are following the convention mentioned after \eqref{pform}.  Since $Y$ is perpendicular to $X$, $Y = (0,y_2,...y_6)\in \mathfrak{p}$.

Computing $0=[X,Y]_\mathfrak{k}$ with the help of \eqref{su3form}, the second row is $$\begin{bmatrix}0 & 0 & -4x_1 y_2 & -3x_1y_3 & -3x_1y_4 & -3x_1 y_5 & -3x_1y_6\end{bmatrix}.$$  Since $x_1\neq 0$, this forces $Y = 0$.  Hence, $[X,Y] = 0$.
\end{proof}

\begin{proposition}\label{equiv}
For $(G,K) = (\Gtwo, SU(3))$, $[\phi(X), \phi(Y)]$ and $[X_\mathfrak{k}, Y_{\mathfrak{k}}]$ are both $0$ iff $[X,Y] $ and $[X_\mathfrak{p}, Y_\mathfrak{p}]$ are both $0$.

\end{proposition}

\begin{proof}
Assume $[\phi(X), \phi(Y)] = [X_\mathfrak{k}, Y_\mathfrak{k}] = 0$.  Expanding, we get $0 = [\phi(X),\phi(Y) ]_\mathfrak{k} = [X_\mathfrak{p}, Y_{\mathfrak{p}}]_\mathfrak{k}$ since $[\mathfrak{k}, \mathfrak{p}]\subseteq \mathfrak{p}$.  By Lemma \ref{almostsym}, $[X_\mathfrak{p}, Y_{\mathfrak{p}}] = 0$.  Thus, $[X_\mathfrak{k}, Y_\mathfrak{p}] + [X_\mathfrak{p}, Y_\mathfrak{k}] = 0$, from which it easily follows that $[X,Y] = 0$.

Conversely, if $[X,Y] = 0$ and $[X_\mathfrak{p}, Y_\mathfrak{p}] = 0$, then $0= [X,Y]_{\mathfrak{k}} = [X_\mathfrak{k}, Y_\mathfrak{k}]$, and then in a similar fashion as above, $[\phi(X),\phi(Y)] = 0$.

\end{proof}

We now turn attention to Wilking's doubling trick, which comes from the observation that the biquotient $\Delta G\backslash G\times G/ 1\times H$ is canonically diffeomorphic to the homogeneous space $G/H$ via the map $(g_1,g_2)\mapsto g_1^{-1}g_2$.  Here, we use the natural action of $\Delta G\times H$ on $G\times G$ given by $(g,h)\ast(g_1,g_2) = (gg_1, gg_2 h^{-1})$.  If we equip $G\times G$ with the product $\langle \cdot, \cdot \rangle_1 + \langle \cdot, \cdot \rangle_1$ of Cheeger metrics which are right $K$-invariant, then this action is isometric, so induces a new metric on $G/H$ which is in general inhomogeneous.

By O'Neill's formula, this new metric is non-negatively curved  as well, and again, \cite{Ta2} Tapp shows that a plane in $G/H$ has zero curvature iff its lift to $G\times G$ has zero curvature, so we can work on $G\times G$.

We first note that under the $\Delta G$ action on $G\times G$, every orbit contains a point of the form $(g, e)$.  As shown in \cite{Ke1}, the image $\mathcal{H}_g$ of the horizontal subspace at such a point, after translating to $(e,e)$ via left multiplication, consists of vectors of the form $$\widehat{X} =\left(\phi^{-1}(-Ad_{g^{-1}} X), \phi^{-1}(X)\right)$$ with $ \langle X,\mathfrak{h}\rangle_0=0.$

Since the metric on $G\times G$ is a product of non-negatively curved metrics, a plane $\operatorname{span}\{\widehat{X}, \widehat{Y}\}$ has zero curvature iff the two planes $$\operatorname{span}\left\{ \phi^{-1}(Ad_{g^{-1}}X), \phi^{-1}(Ad_{g^{-1}} Y)\right\} \text{ and } \operatorname{span}\left\{\phi^{-1}(X), \phi^{-1}(Y)\right\}$$ each have zero sectional curvature.  When $(G,K,H) = (\Gtwo, SU(3), U(2))$, we can combine Propositions \ref{0curv} and \ref{equiv} to find the following characterization of points in $\Delta G\backslash G\times G/1\times H$ at which there are zero-curvature planes.

\begin{theorem}\label{conditions}Suppose $(G,K,H) = (\Gtwo, SU(3), U(2))$.  Then, at a point $(g,e)\in G\times G$, there is a horizontal zero-curvature plane iff there are linearly independent vectors $X,Y\in \mathfrak{g}$ satisfying each of the following three conditions.

1.  $\langle X, \mathfrak{h}\rangle_0 = \langle Y,\mathfrak{h}\rangle_0 = 0$.

2.  $[X,Y] =[X_\mathfrak{k}, Y_\mathfrak{k}]= [X_\mathfrak{p}, Y_\mathfrak{p}] = 0$

3.  $[Ad_{g^{-1}} X, Ad_{g^{-1}} Y] = [(Ad_{g^{-1}} X)_\mathfrak{p}, (Ad_{g^{-1}} Y)_\mathfrak{p}]=0$

\end{theorem}

In fact, since $Ad_{g^{-1}}$ is a Lie algebra isomorphism, the vanishing of the first bracket in Condition $3$ is equivalent to the condition $[X,Y] = 0$.  Also, it is clear that whether or not $X$ and $Y$ satisfy all three conditions only depends on $\operatorname{span}\{X,Y\}$.  We now have an easy corollary of Theorem \ref{conditions}.

\begin{corollary}\label{kpartppart}  The plane spanned by $X$ and $Y$ is spanned by two vectors $X'$ and $Y'$ where $X' = X'_\mathfrak{k}$ and $Y' = Y'_\mathfrak{p}$.
\end{corollary}

\begin{proof}

From Theorem \ref{conditions}, we know $[X_\mathfrak{k}, Y_{\mathfrak{k}}] = 0$.  Since $\langle X,\mathfrak{h}\rangle_0 = \langle Y,\mathfrak{h}\rangle_0 = 0$, we may interpret $X_\mathfrak{k}$ and $Y_\mathfrak{k}$ as tangent vectors in $K/H = SU(3)/U(2) = \mathbb{C}P^2$.  Since $\mathbb{C}P^2$ is positively curved, $[X_\mathfrak{k}, Y_\mathfrak{k}] = 0$ iff $X_{\mathfrak{k}}$ and $Y_\mathfrak{k}$ are linearly dependent.  Hence, by subtracting an appropriate multiple of $X$ from $Y$, we get a new vector $Y'$ with $Y'_{\mathfrak{k}} = 0$.  Note that $\operatorname{span}\{X, Y\} = \operatorname{span}\{X, Y'\}$, so $\operatorname{span}\{X, Y'\}$ satisfies all the conditions of Theorem \ref{conditions}.

\

Now, we also have the equation $[X_\mathfrak{p}, Y'_{\mathfrak{p}}] = 0$.  Interpreting these as tangent vectors on the positively curved $G/K = \Gtwo/SU(3) = S^6$, we see the bracket vanishes iff $X_\mathfrak{p}$ and $Y'_{\mathfrak{p}}$ are linearly dependent.  Then, subtracting an appropriate multiple of $Y'$ from $X$, we obtain a new vector $X'$ with $X'_{\mathfrak{p}} = 0$, and $\operatorname{span}\{X', Y'\} = \operatorname{span}\{X,Y\}$.

\end{proof}

Consider the action of $H'\times K$ on $G\times G$ given by $(h,k)\ast(g_1,g_2) = (g_1 k^{-1}, g_2 h^{-1})$.  This action is isometric since $\langle \cdot, \cdot\rangle_1$ is right $K$-invariant.  This action moves a point of the form $(g_1,e)\in G\times G$ to $(g_1 k^{-1}, h^{-1})$, which, after left multiplication by $(h,h)$, becomes $(hg_1 k^{-1} , e)$.  From Proposition \ref{niceform}, we now have the following corollary.

\begin{corollary}\label{niceform2}  Every point $[(g,e)]\in \Delta G\backslash G\times G/1\times H\cong G/H$ is isometrically equivalent to a point $[(g',e)]$ with $g'\in \mathcal{F}$ via the $H'\times K$ action.

\end{corollary}

The following proposition allows us to restrict attention to $\mathcal{F}$.  

\begin{proposition}\label{finalprop} Suppose the set of points in $\mathcal{F}$ with no zero-curvature planes with respect to $\langle \cdot, \cdot\rangle_2$ is dense in $\mathcal{F}$.  Then $G/H$ has almost positive curvature.

\end{proposition}
\begin{proof}

Since having positive curvature is an open condition, we need only show the set of points in $G/U$ with no zero-curvature planes is dense.  So, let $V\subseteq \Delta G\backslash G\times G/1\times H$ be a non-empty open set and suppose $[(g,e)]\in V$.  We need to find a point in $V$ with no zero-curvature planes.

By Corollary \ref{niceform2}, there is an isometry $f$ of $\Delta G\backslash G\times G/1\times H$ for which $f([(g,e)]) = [(g',e)]$ with $g'\in \mathcal{F}$.  Then $f(V)\cap \mathcal{F}$ is a neighborhood of $[(g',e)]$ in $\mathcal{F}$, so by assumption, there is a point $[(g'',e)] \in f(V)\cap \mathcal{F}$ with no zero-curvature planes.  Then, $f^{-1}[(g'',e)]\in V$ has no zero-curvature planes.

\end{proof}

\section{Almost Positive Curvature}\label{apc}

In this section, we complete the proofs of Theorems \ref{main} and \ref{main2}, relying on Theorem \ref{conditions} and Corollary \ref{niceform2}.  We note that $G/H$ is locally isometric to $G/H'$, so it is enough to show that $G/H$ is almost positively curved.

We work at a point $(g,e)\in G\times G$ with $g\in \mathcal{F}$, assuming that $[(g,e)]$ has at least one zero-curvature plane.  Thus, by Theorem \ref{conditions}, there are linearly independent $X,Y\in \mathfrak{g}$ which satisfy all the conditions of Theorem \ref{conditions}.  From Corollary \ref{kpartppart}, we can assume without loss of generality that $X = X_\mathfrak{k}$ and $Y = Y_\mathfrak{p}$.  Following the convention after \eqref{pform}, we write $Y = (y_1,...,y_6)$.  Now, $\langle Y,\mathfrak{h}\rangle_0 = 0$ automatically, but for $X$, the condition $\langle X,\mathfrak{h}\rangle_0 = 0$ forces $X$ to have the following form.
$$\begin{bmatrix}
0&0&0&0&0&0&0 \\ 
0&0&0& x_1 & -x_2 & x_3 & -x_4 \\ 
0&0&0& x_2 & x_1 & -x_4 & -x_3 \\ 
0& -x_1 & -x_2 &0&0&0&0 \\ 
0& x_2 & -x_1 &0&0&0&0 \\ 
0& -x_3 & x_4 &0&0&0&0 \\ 
0& x_4 & x_3 &0&0&0&0 
\end{bmatrix}$$ 

As $Y = Y_\mathfrak{p}$, it follows that $[X,Y] = [X,Y]_\mathfrak{p}$, so that $[X,Y] =0$ iff the first row of $[X,Y]$ is $0$.  We again recall the identification of $\mathfrak{p}$ with $\mathbb{R}^6$ mentioned after \eqref{pform}, though we will henceforth write elements in $\mathfrak{p}$ as column vectors for convenience.  With respect to this identification, a direct calculation now shows that $[X,Y]$ is given by \begin{equation}\label{bracket0}\frac{1}{2}\begin{bmatrix}x_1 y_3 - x_2 y_4 +x_3y_5 - x_4 y_6\\ x_1 y_4 +x_2 y_3 - x_3 y_6 - x_4 y_5\\ -x_1 y_1 -x_2 y_2\\ -x_1 y_2 + x_2 y_1\\ -x_3 y_1 + x_4y_2 \\ x_3 y_2 + x_4 y_1 \end{bmatrix}.\end{equation}

Focusing on the third and fourth entries, we view these as a linear system with variable $y_1$ and $y_2$ and coefficients given in terms of $x_1$ and $x_2$.  Then the coefficient matrix has determinant $\frac{1}{4}(x_1^2 + x_2^2)$.  Likewise, focusing on the last two entries, we see the corresponding coefficient matrix has determinant $-\frac{1}{4}(x_3^2 + x_4^2)$.  Since $X\neq 0$, at least one of the $x_i$ is non-zero.  Hence, the condition $[X,Y] = 0$ forces $y_1 = y_2 = 0$, that is, $$Y = \begin{bmatrix}0 & 0 & 0 & 2y_3 & 2y_4 & 2y_5 & 2y_6 \\ 0 & 0 & 0 & -y_6 & y_5 & -y_4 & y_3\\ 0 & 0 & 0 & y_5 & y_6 & -y_3 & -y_4 \\ -2y_3 & y_6 & -y_5 & 0 & 0 & 0 & 0\\ -2y_4 & -y_5 & -y_6 & 0 & 0 & 0 & 0\\ -2y_5 & y_4 & y_3 & 0 & 0 &0 & 0 \\ -2y_6 & -y_3 & y_4 & 0 & 0 & 0 & 0 \end{bmatrix}.$$

\

Now, we consider condition $3$ of Theorem \ref{conditions}, ${[(Ad_{g^{-1}} X)_{\mathfrak{p}}, (Ad_{g^{-1}} Y)_{\mathfrak{p}}] = 0}$.  Since $G/K = S^6$ is positively curved, this condition holds iff one of $(Ad_{g^{-1}} X)_{\mathfrak{p}}$ or $(Ad_{g^{-1}} Y)_{\mathfrak{p}}$ is $0$, or if one is a non-zero multiple of the other.  Recall that we are assuming $g\in\mathcal{F}$, so $(Ad_{g^{-1}} X)_{\mathfrak{p}}$ is given by the column vector
$$\frac{1}{2}\begin{bmatrix}
0 \\ 
2x_2\sin\phi\cos\phi\sin\theta \\
-x_1\sin\theta \\
x_2(2\cos^2\phi-1)\sin\theta\cos\theta+x_3\cos\phi\sin^2\theta \\
x_2(2\cos^2\phi-1)\sin^2\theta - x_3\cos\phi\sin\theta\cos\theta \\
x_4\cos\phi\sin\theta
\end{bmatrix}.$$
Similarly, $(Ad_{g^{-1}} Y)_{\mathfrak{p}}$ is given by
$$\frac{1}{2} \begin{bmatrix}
2y_3\sin\phi \\ 
2y_4\sin\phi\cos\theta \\
2y_3\cos\phi\cos\theta + y_6\sin\theta \\
y_4\cos\phi(3\cos^2\theta-1) - 3y_5\sin\theta\cos\theta \\
3y_4\cos\phi\sin\theta\cos\theta + y_5(3\cos^2\theta-1)  \\
-y_3\cos\phi\sin\theta + 2y_6\cos\theta
\end{bmatrix}.$$

\

We will initially assume $\theta \in (0,\pi/2)$ and $\phi\in [0,\pi/2)$; we will address the remaining end points later in this section. It is clear from the above expressions for $(Ad_{g^{-1}} X)_\mathfrak{p} $ and $(Ad_{g^{-1}} Y)_\mathfrak{p}$ that these vectors are $0$ iff $X$ or $Y$ are $0$.  Since $X$ and $Y$ are linearly independent, this is a contradiction.  Therefore, we may assume one is a non-zero multiple of the other. Since we can scale these vectors without changing the plane they span, we may assume $(Ad_{g^{-1}} X)_\mathfrak{p} = (Ad_{g^{-1}} Y)_\mathfrak{p}$.

For now, assume $\phi > 0$.  By comparing the entries in the first row, it is clear that $y_3$ must equal 0. We can then solve the equations formed by rows three, two, and six for $x_1$, $x_2$, and $x_4$ respectively to obtain the following:
$$x_1 = -y_6 \hspace{3em} x_2 = \frac{y_4\cos\theta}{\cos\phi\sin\theta} \hspace{3em} x_4 = \frac{2y_6\cos\theta}{\cos\phi\sin\theta}.$$
Substituting these values into the equations formed by row four and row five, we can solve for $x_3$ in two ways.  Thus, $$x_3 = \frac{y_4(\cos^2\theta - \cos^2\phi \sin^2\theta) - 3y_5(\cos\phi\cos\theta\sin\theta)}{\cos^2\phi \sin^2\theta}$$ and $$x_3 = \frac{-y_4\cos\theta\sin\theta(\cos^2\phi +1) + y_5 \cos\phi(1-3\cos^2\theta)}{\cos^2\phi \cos\theta\sin\theta}.$$ Then by setting the right hand sides of these two equations equal to each other, we see
$$y_4 = \frac{y_5\cos\phi\sin\theta}{\cos\theta}.$$

Now we return to the condition $[X,Y]=0$ given in \eqref{bracket0}.  After making all of the substitutions above, we see all entries of $[X,Y]$ are automatically $0$ except for the first.  Thus, a zero-curvature planes exist iff the following equation holds,
$$\frac{-4(y_5^2\sin^2\phi+y_6^2)-4y_5^2\cos^2\phi}{\cos\phi\cos\theta\sin\theta}=0.$$ 
Clearly this can only occur if $y_5=y_6=0$, causing $Y$ to be identically zero. This is a contradiction of the independence of $X$ and $Y$. Thus, when $\theta,\phi \in (0,\pi/2)$ we have no zero-curvature planes, implying positive curvature. This set of points is dense in $\mathcal{F}$, and therefore by Proposition \ref{finalprop}, $G/H$ has almost positive curvature completing the proof of Theorem \ref{main}.

\

If $\phi = 0$, the first row of the equation $(Ad_{g^{-1}} X)_\mathfrak{p} = (Ad_{g^{-1}} Y)_\mathfrak{p}$ no longer provides information about $y_3$.  Instead, we proceed by solving the equations formed by rows $3,4,5,$ and $6$ for the $x_i$ variables.  This gives $$\begin{bmatrix} x_1 \\ x_2\\ x_3\\ x_4\end{bmatrix} = \frac{1}{\sin\theta} \begin{bmatrix} -2y_3\cos \theta - y_6\sin\theta\\ 2y_4\cos\theta -y_5\sin\theta\\ - y_4\sin\theta-2y_5\cos\theta \\ - y_3\sin\theta+2y_6\cos\theta  \end{bmatrix}.$$

Then, making these substitutions into \eqref{bracket0}, we see that $[X,Y] =  0$ iff $$\frac{-4\cos\theta(y_3^2+y_4^2+y_5^2+y_6^2)}{\sin\theta} = 0.$$  This implies $Y = 0$, so $X$ and $Y$ are not independent.  This contradiction proves that we have positive curvature when $\phi =  0$ and $\theta\in(0,\pi/2)$.

\

Now we prove Theorem \ref{main2} by focusing on the remaining cases, where $\theta \in\{0,\pi/2\}$ or $\phi = \pi/2$.  Notice that when $\theta=0$, then $(Ad_{g^{-1}} X)_{\mathfrak{p}} = 0$, and we have zero-curvature planes obtained by, e.g., setting $x_1,x_2,y_5,$ and $ y_6$ equal to $0$.  When $\theta= \pi/2$ it is easy to see that making the substitutions $x_2=y_3=x_4=y_5=0$, $x_1=-y_6$, and $x_3=-y_4$ satisfies all the conditions in Theorem \ref{conditions}.  Hence, there are zero-curvature planes at these points as well.

For the case when $\phi = \pi/2$, by setting $x_1=x_2 = 0$, we make $(Ad_{g^{-1}} X)_{\mathfrak{p}}=0$.  Then one easily sees that $[X,Y]=0$ if $y_5 = y_6 = 0$.  Then any non-zero choice of $x_3, x_4, y_3$, and $y_4$ gives a zero-curvature plane.

In summary, we have shown the following theorem.

\begin{theorem}\label{0curvdesc}

A point $(g,e)\in G\times G$ with $g = (g)_{ij}\in \mathcal{F}$ projects to a point having at least one zero-curvature plane iff $\theta = 0$, $\theta = \pi/2$, or $\phi = \pi/2$.

\end{theorem}

From the discussion following the proof of Proposition \ref{niceform}, the points $(g,e)\in G\times G$ which project to zero-curvature planes have either $g_{11} = 0$ or $g_{21} = g_{31} = 0$.  Recall the diffeomorphism $\Delta G \backslash G\times G/1\times H\cong G/H$ induced from the map $G\times G\rightarrow G$ with $(g_1,g_2)\mapsto g_1^{-1} g_2$.  Under this diffeomorphism, we see that the points in $G$ which project to zero-curvature planes in $G/H$ all have $g_{11} = 0$ or $g_{12} = g_{13} = 0$.  Thus, we have proved Theorem \ref{main2}.

\section{The Topology of the Zero-Curvature Points}\label{top}

In this section, we investigate the topology of the set of points in $\Gtwo/U(2)$ which have at least one zero-curvature plane with respect to $\langle \cdot,\cdot \rangle_2$.  

We recall $Z_1 = \{g\in \Gtwo: g_{12} = g_{13} = 0\}$ and  ${Z_2 = \{g\in \Gtwo: g_{11} = 0\}}.$  By Theorem \ref{main2} a point $g\in G = \Gtwo$ projects to a point with at least one zero-curvature plane iff $g\in Z_1\cup Z_2$.  For $g\in G$, we will use the notation $\overline{g}$ to denote its image in $\Gtwo/U(2) = G/H$.

We begin with an alternative proof to that found in \cite{Ker1}, showing that $G/H$ is diffeomorphic to $Gr_2\!\left(\mathbb{R}^7\right)$.  Recall that $g_{\bullet k}$ refers to the $k$th column of $g\in G$.

\begin{proposition} The map $\psi_H:G\rightarrow Gr_2\!\left(\mathbb{R}^7\right)$ which sends a matrix $g\in G$ to the plane with ordered orthonormal basis $\{g_{\bullet 2}, g_{\bullet 3}\}$ descends to a diffeomorphism $G/H\rightarrow Gr_2\!\left(\mathbb{R}^7\right)$.  Further, $\psi_H$ maps $\overline{Z}_1$ diffeomorphically onto $Gr_2\!\left(\mathbb{R}^6\right)$, where $\mathbb{R}^6 \cong i^\bot\subseteq \Imo.$

\end{proposition}

\begin{proof}

By Theorem \ref{G2desc}, $\psi_H$ is surjective, so we need only show that it descends to an injective map $G/H \rightarrow Gr_2\!\left(\mathbb{R}^7\right)$.

Let $P$ denote the oriented plane with oriented basis $\{j,k\}$.  Then $\psi_H(g) = \psi_H(g')$ iff $gP = g'P$ which holds  iff $g^{-1}g'P = P$.  Thus, $\psi_H(g) = \psi_H(g')$ iff $g^{-1}g'\in H$, that is, iff $gH =g'H$.

\

Finally, we show that $\psi_H$ restricts to a surjective map from $Z_1$ to $Gr_2\!\left(\mathbb{R}^6\right)$.  First note that for $g\in Z_1$, $g_{12} = g_{13} = 0$, so the columns $g_{\bullet 2}$ and $g_{\bullet 3}$ are both perpendicular to $i$.  It follows that $\psi_H(g)\subseteq Gr_2\!\left(\mathbb{R}^6\right)$.

On the other hand, given a $2$-plane $Q\in Gr_2\!\left(\mathbb{R}^6\right)$, choose an oriented orthonormal basis $\{q_2, q_3\}$ for it.  By Theorem \ref{G2desc}, there is a unique matrix $g\in G$ with $g_{\bullet 2} = q_2$, $g_{\bullet 3} = q_3$, and $g_{\bullet 4}  =(1,0,...,0)^t$.  Clearly $\psi_H(g) = Q$.  Additionally, since $q_2, q_3 \bot i$, it follows that $g_{12} = g_{13} = 0$, so $g\in Z_1$.

\end{proof}

In a similar fashion, the map $\psi_K$ which sends $g\in G$  to $g_{\bullet 1}\in S^6$ descends to a diffeomorphism $G/K\cong S^6$.  In fact, since $g\in G$ implies $g_{\bullet 2} g_{\bullet 3} = g_{\bullet 1}$, the maps $\psi_H$ and $\psi_K$ give a bundle isomorphism \begin{diagram}[LaTeXeqno]\label{bund}K/H &\rTo&  G/H & \rTo & G/K \\  \dTo & & \dTo^{\psi_H} & & \dTo^{\psi_K}\\ \mathbb{C}P^2 &\rTo& Gr_2\!\left(\mathbb{R}^7\right)& \rTo^{\pi} & S^6\end{diagram} where the projection $\pi:Gr_2\!\left(\mathbb{R}^7\right)\rightarrow S^6$ maps a plane with oriented orthonormal basis $\{g_{\bullet 2}, g_{\bullet 3}\}$ to $g_{\bullet 2} g_{\bullet 3}$.

We can now prove the following.

\begin{proposition}

The subspace $\overline{Z}_2$ of $G/H$ is diffeomorphic to $\mathbb{C}P^2\times S^5$.

\end{proposition}

\begin{proof}

Let $S^5\subseteq S^6\subseteq \Imo$ denote the equatorial $S^5$ with $i$-coordinate equal to $0$.  Since $g\in Z_2$ iff $g_{11} = 0$, we see that $\overline{Z}_2 = \pi^{-1}(S^5)$.

Pulling back the bundle \eqref{bund} along the inclusion $S^5\rightarrow S^6$, we get a bundle $\mathbb{C}P^2\rightarrow Z_2\rightarrow S^5$ with structure group $K = SU(3)$.  Since $\pi_4(SU(3)) \cong \pi_4(U(3))$ is in the stable range, it vanishes by Bott periodicity.  Thus, every principal $SU(3)$ bundle over $S^5$ is trivial, and hence, so is every associated bundle.  Thus $\overline{Z}_2$ is diffeomorphic to $\mathbb{C}P^2\times S^5$.

\end{proof}

We may now determine the structure of $\overline{Z}_1\cap \overline{Z}_2$. Recall that the Aloff-Wallach Space $W_{1,-1}$ is the homogeneous space $SU(3)/\{\diag(z,\overline{z},1)\}$ where $z\in S^1$.  The subgroup $\diag(z,\overline{z}, 1)$ is conjugate to the subgroup $\diag(R(\alpha), 1)$ and hence, $W_{1,-1}$ and $SU(3)/\{\diag(R(\alpha),1)\}$ are canonically diffeomorphic.

\begin{proposition}
The subspace $\overline{Z}_1\cap \overline{Z}_2$ of $G/H$ is diffeomorphic to the Aloff-Wallach space $W_{1,-1} = W_{1,0}$.

\end{proposition}

\begin{proof}

Let $g\in Z_1\cap Z_2$, so $g_{11} = g_{12} = g_{13} = 0$.  We recall that because $g\in \Gtwo$, $g_{\bullet 2} g_{\bullet 3} = g_{\bullet 1}$.  In particular, $g_{\bullet 2} g_{\bullet 3} \bot i$.  Since the octonions are a normed algebra, right multiplication by any element of unit length is an isometry, so we see that $(g_{\bullet 2} g_{\bullet 3}) g_{\bullet 3} \bot  i g_{\bullet 3}$.  The octonions are alternative, so $(g_{\bullet 2} g_{\bullet 3})  g_{\bullet 3} = g_{\bullet 2}(g_{\bullet 3})^2 = -g_{\bullet 2}$.  Since this argument is reversible, we see that $g\in Z_1\cap Z_2$ iff $g_{\bullet 2}$ is perpendicular to both $g_{\bullet 3}$ and $i g_{\bullet 3}$.

If we identify $g_{\bullet 2 } = (0, g_{22}, g_{23},.., g_{27})^t$ with the complex $3$-tuple $$\widetilde{g}_2 = (g_{22} + i g_{23}, g_{24} + i g_{25}, g_{26} + ig_{27})^t, $$ then it is easy to verify that the octonion multiplication $ig_{\bullet 2}$ is the equivalent to the complex multiplication $i\,\widetilde{g}_2$.  In particular, the vectors $\widetilde{g}_2,\widetilde{g}_3\in \mathbb{C}^3$ are orthogonal with respect to the usual Hermitian inner product on $\mathbb{C}^3$.  It follows that there is a unique matrix $A = A\left(\widetilde{g}_2,\widetilde{g}_3\right)\in SU(3)$ with columns $\widetilde{g}_2$ and $\widetilde{g}_3$.

\

Now, consider the smooth map $$f:\overline{Z}_1\cap\overline{Z}_2\rightarrow SU(3)/\{\diag(R(\alpha),1)\}$$ given by mapping $[g]$ to $\left[A\left(\widetilde{g}_2,\widetilde{g}_3\right)\right]$.  To see this map is well-defined, we first recall that $H\cong U(2)$ consists of elements of the block diagonal form $\diag(1,R(\alpha), B)$ where $R(\alpha)$ denotes the usual $2\times 2$ rotation matrix, and $B$ is a $4\times 4$ matrix contained in a $U(2)\subseteq SO(4)$.  For such an element in $H$, we have $$[g\,\diag(1,R(\alpha), B)]\mapsto [A(\widetilde{g}_2, \widetilde{g}_3) R(\alpha)] = [A(\widetilde{g}_2,\widetilde{g}_3)].$$

The inverse of $f$ can be constructed as follows.  Given $A\left(\widetilde{g}_2,\widetilde{g}_3\right)\in SU(3)$, the columns $\widetilde{g}_2$ and $\widetilde{g}_3$ are orthogonal with respect to the usual Hermitian inner product on $\mathbb{C}^3$.  Hence, the vectors $g_{\bullet 2}, g_{\bullet 3}\in i^\bot\subseteq \Imo$ are orthogonal, as are $g_{\bullet 3}$ and $ig_{\bullet 2}$.  It follows that $g_{\bullet 2} g_{\bullet 3} \bot i$ as well.  Now, from Theorem \ref{G2desc} there is a unique matrix $g$ in $\Gtwo$ with columns $g_{\bullet 2}$, $g_{\bullet 3}$, and $g_{\bullet 4} =  (1,0,...,0)^t$.  Since $g_{\bullet 2} g_{\bullet 3}\bot i$, $B\in Z_1\cap Z_2$.  The mapping $A\mapsto g$ descends to $$f^{-1}:SU(3)/\{\diag(R(\alpha),1)\}\rightarrow \overline{Z}_1\cap \overline{Z}_2.$$

\end{proof}

This completes the proof of Theorem \ref{main3}.


\textbf{Acknowledgements:  }We would like thank Wolfgang Ziller for helpful comments on an earlier draft of this paper.  We are also grateful to acknowledge support from the Blankenship Undergraduate Research Endowment.\label{ackref}

\end{document}